\documentclass[11pt]{article}
\usepackage{graphicx}
\usepackage{makeidx}

\usepackage[margin=1in]{geometry}

\usepackage{amsmath}
\usepackage{amsthm}
\usepackage{amsfonts}
\usepackage{amssymb}
\usepackage{epsfig}
\usepackage{tabularx}
\usepackage{latexsym}
\newtheorem{lemma}{Lemma}
\newtheorem{observation}[lemma]{Observation}
\def\qed{$\Box$}

\newtheorem{result}{Result}

\date{}

\begin{document}

\title{On the construction of non-$2$-colorable uniform hypergraphs}
\author{Jithin Mathews 
\footnote{IIT Guwahati, Assam-781039,
India. Email: m.jithin@iitg.ernet.in} 
\and Manas Kumar Panda \footnote{IIT Guwahati, Assam-781039, India. Email: manas.panda@iitg.ernet.in}
\and Saswata Shannigrahi \footnote{IIT Guwahati, Assam-781039, India. Email: saswata.sh@iitg.ernet.in}
\footnote{Corresponding author}
}
\maketitle


\begin{abstract}
The problem of $2$-coloring uniform hypergraphs has been extensively studied over the last few decades. An $n$-uniform hypergraph is not 
$2$-colorable if its vertices cannot be colored with two colors, Red and Blue, such that every hyperedge contains Red as well as Blue vertices. The least possible number of hyperedges in an $n$-uniform hypergraph which is not $2$-colorable is denoted by $m(n)$. In this paper, we consider the problem of finding an upper bound on $m(n)$ for small values of $n$. We provide constructions which improve the existing results for some such values of $n$. We obtain the first improvement in the case of $n=8$.
\vspace{1mm}

\noindent \textbf{Keywords:} Uniform Hypergraph; Property B; Hypergraph Coloring
\end{abstract}

\section{Introduction}
\label{introduction}

A hypergraph is said to have property B if there exists a proper subset $\textit{S}$ of its vertices such that every hyperedge of the hypergraph contains vertices from both $\textit{S}$ and $\textit{$\overline{S}$}$, the complement of $S$. In other words, if the vertices of a hypergraph can be colored using Red and Blue colors such that every hyperedge of the hypergraph contains Red as well as Blue vertices, then that hypergraph is said to have {\it Property B}. It is known that there exist $2$-colorable hypergraphs with arbitrarily high number of hyperedges. For example, Lov\'{a}sz mentions the following in Problem $13.33$ of \cite{Lovasz}: if there are no two hyperedges with exactly one common vertex in a hypergraph, then the hypergraph is $2$-colorable. 

The least possible number of hyperedges in an $n$-uniform hypergraph which does not have Property B is denoted by $m(n)$. One motivation to study such hypergraphs is the following relationship between non-$2$-colorable hypergraphs and unsatisfiable CNF formulas. In fact, constructing a non-$2$-colorable $n$-uniform hypergraph $H$ with $x$ hyperedges is equivalent to constructing an unsatisfiable monotone $n$-CNF with $2x$ clauses. For a given $n$-uniform hypergraph $H$, let $H'$ denote the $n$-CNF obtained by adding clauses $C_{e}:= (x_{1} \vee x_{2} \dots \vee x_{n})$ and $\bar C_{e}:= (\bar x_{1} \vee \bar x_{2} \dots \vee \bar x_{n})$ for every hyperedge $e$ = $\{x_{1},x_{2}, \dots,x_{n}\} \in H$. Note that $H'$ is monotone, i.e., every clause of $H'$ either contains only non-negated literals or only negated literals. It can be easily seen that every $2$-coloring $\chi$ of $H$ yields a satisfying assignment $\alpha$ of $H'$ and vice versa.

A lot of work has been done to find a lower bound on $m(n)$. Erd\H{o}s \cite{erdos1} showed that $m(n) = \Omega({2^n})$, which was improved to $\Omega({n^{1/3 - o(1)} 2^n})$ by Beck \cite{beck}. It was further improved by  Radhakrishnan and Srinivasan \cite{radhakrishnan} to the currently best known lower bound $m(n)=\Omega({\sqrt{\frac{n}{\ln n}} 2^n})$. In the other direction, Erd\H{o}s \cite{erdos} used the probabilistic method to show that $m(n) = O(n^2 2^n)$. However, as it is typical in such constructions, this method did not provide any explicit construction that matches this bound. For a general $n$, only a few explicit constructions of non-2-colorable uniform hypergraphs are known. Recently, Gebauer \cite{geba} provided a construction which produces a non-$2$-colorable $n$-uniform hypergraph containing $2^n \cdot 2^{O(n^{\frac{2}{3}})}$ hyperedges for a sufficiently large $n$. This is the asymptotically best-known  constructive upper bound on $m(n)$ till now. Some constructions for small values of $n$ are mentioned in \cite{toft}. However, all the known constructions give an upper bound on $m(n)$ which is asymptotically far from Erd\H{o}s's non-constructive upper bound.

Finding $m(n)$ for smaller values of $n$ is also an extensively studied topic. It is straightforward to see that $m(2) = 3$ (triangle graph) and $m(3) = 7$ (Fano plane \cite{klein}). But, finding $m(n)$ for $n\geq 4$ remained an open problem for a long time. The construction of Abbott and Hanson \cite{Abbott} proved that $m(4) \leq  24$. Seymour \cite{seymour} improved upon this construction to show that $m(4) \leq  23$. In the other direction, Manning \cite{manning} provided a proof that $m(4) \geq  21$. Recently, \"{O}sterg\r{a}rd \cite{patric} has proved that $m(4) \geq 23$. 
This shows that $m(4) = 23$, and settles this problem for $n = 4$. The problem of finding the exact value of $m(n)$ for $n\geq 5$ still remains an open problem.

The best known upper bound on $m(n)$ for small values of $n$ are obtained from three constructions where each construction provides a recurrence relation. We describe these constructions below.

The first construction is due to Abbott and Moser \cite{abb&mos} for all composite values of $n$. They proved that if $n$ = $a \cdot b$ such that $a$ and $b$ are integer factors of $n$, then $m(n) \leq m(a) \cdot (m(b))^a$. Let us discuss their construction in brief. Take an $a$-uniform hypergraph giving the best known value for $m(a)$ and replace the vertices of this hypergraph with different $b$-uniform hypergraphs that gives the best known value for $m(b)$. Note that the vertex set of each of these $b$-uniform hypergraphs must be disjoint from each other. In order to form an $n$-uniform hyperedge, the hyperedges of these $a$ different $b$-uniform hypergraphs are combined with each other in every possible way. Therefore, the $n$-uniform hypergraph has a total of $m(a) \cdot (m(b))^a$ hyperedges. A more detailed description of this construction and the proof for the claim that the hypergraph formed by this construction is not $2$-colorable can be found in \cite{abb&mos}.

The second construction is by Abbott and Hanson \cite{Abbott} which gives the recurrence $m(n) \leq 2^{n-1} + n \cdot m(n-2)$ when $n$ is odd. The third construction is by Toft \cite{toft2}, who obtained the following bound when $n$ is even: $m(n) \leq 2^{n-1} + {n \choose {n / 2}} / {2} + n \cdot m(n-2)$. In Section $2$, we discuss these two constructions in some details. Exoo \cite{exoo} tried to improve the upper bounds on $m(n)$ for small values of $n$ using computer programs. However, his constructions did not improve any of the known upper bounds obtained from the recurrences mentioned above \cite{toft}. 

\subsection{Our contribution} In Section $3$, we provide a construction that improves the currently best known upper bounds on $m(n)$ for some small values of $n$, starting from $n = 13$. 

\begin{result}
We provide a construction which shows that $m(n) \leq (n+1) \cdot 2^{n-2} + \linebreak (n-1) \cdot m(n-2)$ when $n$ is odd, and $m(n) \leq (n+1) \cdot 2^{n-2} + {n \choose {n / 2}}/{2} + (n-1) \cdot \left( m(n-2) + {n-2 \choose {(n-2)}/{2}} \right)$ when $n$ is even.
\end{result}

\noindent However, in Section $4$, we provide another construction that improves the construction given in Section $3$.

\begin{result}
We provide a construction which shows that $m(n) \leq (n+4) \cdot 2^{n-3} + \linebreak (n-2) \cdot m(n-2)$ when $n$ is odd, and $m(n) \leq (n+4) \cdot 2^{n-3} +
n \cdot {n-2 \choose {(n-2)/2}}/{2} + {n \choose {n / 2}}/{2} + (n-2) \cdot m(n-2)$ when $n$ is even.
\end{result}

\noindent Using this recurrence relation, the first improvement is obtained when $n=11$. We show that $m(11) \leq 25449$, which improves the currently best known bound $m(11) \leq 27435$. \\\\
In Section $5$, we provide a construction for $n=8$ that improves the currently best known bound $m(8) \leq 1339$.
\begin{result}
$m(8) \leq 1269.$
\end{result}

\section{The construction of Abbott and Hanson \cite{Abbott}, and Toft \cite{toft2}}

In this section, we discuss the construction made by Abbott and Hanson \cite{Abbott}, which was further improved upon by Toft \cite{toft2}. In fact, Abbott and Hanson's \cite{Abbott} construction is good for odd values of $n$, while Toft \cite{toft2} improved their construction for even values of $n$. Now let us discuss their construction, which we call as the {\it AHT construction}, in detail.\footnote{Note that the variables used in a section is valid only inside that section.}

For a given $n \geq 3$, let us consider an $(n-2)$-uniform hypergraph producing the best known upper bound for $m(n-2)$. We denote this hypergraph as a {\it core hypergraph}. Let $m_{n-2}$ be the number of hyperedges present in this core hypergraph $C =(X, Y)$. The set of hyperedges is denoted by 
$Y = \{e_1, e_2, e_3, \ldots, e_{m_{n-2}}\}$, where $e_i$ is the $i^{th}$ hyperedge in the core hypergraph $C$. 
Let $U = \{u_1, u_2, u_3, \ldots, u_n\}$ and $V = \{v_1, v_2, v_3, \ldots, v_n\}$ denote two disjoint set of vertices, each of them disjoint from $X$, the set of vertices in the core hypergraph $C$.
For each $1 \leq i \leq n$, we call $u_i$ and $v_i$ to be a pair of {\it matching vertices}. For any $K =  \{a_1, a_2, \ldots a_k\}$ which is a proper subset of $\{1, 2, \ldots, n\}$ such that $1 \leq a_1 < a_2 \ldots < a_k \leq n$, we denote $U_K = \{u_{a_1}, u_{a_2}, \ldots, u_{a_k}\}$ and $V_K = \{v_{a_1}, v_{a_2}, \ldots, v_{a_k}\}$. We also define 
${\bar{U}}_K = U \setminus U_K$ and ${\bar{V}}_K = V \setminus V_K$.

Let us now define the construction of the non-$2$-colorable $n$-uniform hypergraph $H$. The vertex set of this hypergraph is $X \cup U \cup V$, and the following hyperedges are present in $H$:\\

$(i)$ $\{u_i\} \cup e_j \cup \{v_i\}$ for every pair of $i, j$ satisfying $1 \leq i \leq n$ 
and $1 \leq j \leq m_{n-2}$. 

$(ii)$ $U$

$(iii)$ $U_K \cup {\bar{V}}_K$ for every $K$ such that $|K|$ is odd, and $1 \leq |K| \leq \lfloor n/2 \rfloor$.

$(iv)$ $V_K \cup {\bar{U}}_K$ for every $K$ such that $|K|$ is even, and $2 \leq |K| \leq \lfloor n/2 \rfloor$. \\

\noindent Note that the total number of hyperedges in this hypergraph is \\

$1 + \binom{n}{1} + \binom{n}{2} + \ldots + \binom{n}{\lfloor n/2 \rfloor} + n \cdot m_{n-2}$.
\\

\noindent When $n$ is odd, this shows that $m(n) \leq 2^{n-1} + n \cdot m(n-2)$.

\noindent When $n$ is even, this shows that $m(n) \leq 2^{n-1} + {n \choose {n / 2}}/{2} + n \cdot m(n-2)$.
\begin{table}
\caption{The currently best known upper bound on $m(n)$ for small values of $n$}
\begin{center}
\begin{tabular}{|c|c|c|}
\hline
$n$ & $m(n)$ & Obtained by the recurrence relation\\
\hline
1& $m(1) = 1$&Single vertex\\
2&$m(2) = 3$&Triangle graph\\
3&$m(3) = 7$&Fano plane \cite{klein}\\
4&$m(4) = 23$&$m(4) \leq 2^3 + {4 \choose 2} / {2} + 4 \cdot m(2)$ and $\cite{patric}$\\
5&$m(5) \leq 51$&$m(5) \leq 2^4 + 5 \cdot m(3)$\\
6&$m(6) \leq 147$&$m(6) \leq m(2) \cdot m(3)^2$\\
7&$m(7) \leq 421$&$m(7) \leq 2^6 + 7 \cdot m(5)$\\
8&$m(8) \leq 1339$&$m(8) \leq 2^7 + {8 \choose 4} / {2} + 8 \cdot m(6)$\\
9&$m(9) \leq 2401$&$m(9) \leq m(3) \cdot m(3)^3$\\
10&$m(10) \leq 7803$&$m(10) \leq m(2) \cdot m(5)^2$\\
11&$m(11) \leq 27435$&$m(11) \leq 2^{10} + 11 \cdot m(9)$\\
12&$m(12) \leq 64827$&$m(12) \leq m(2) \cdot m(6)^2$\\
13&$m(13) \leq 360751$&$m(13) \leq 2^{12} + 13 \cdot m(11)$\\
14&$m(14) \leq 531723$&$m(14) \leq m(2) \cdot m(7)^2$\\
15&$m(15) \leq 857157$&$m(15) \leq m(5) \cdot m(3)^5$\\
16&$m(16) \leq 5378763$&$m(16) \leq m(2) \cdot m(8)^2$\\
17&$m(17) \leq 14637205$&$m(17) \leq 2^{16} + 17 \cdot m(15)$\\

\hline
\end{tabular}
\end{center}
\end{table}

In Table $1$, we show the upper bound on $m(n)$ for some small values of $n$ using the currently best known constructions for small $n$. In the following section, we provide a better construction which improves the upper bound of $m(n)$ for small values of $n$.

\section{The first improvement over the AHT construction}

For a given $n \geq 3$, let us consider an $(n-2)$-uniform hypergraph producing the best known upper bound for $m(n-2)$. Note that this hypergraph is the {\it core hypergraph} as described in the previous section. Let $m_{n-2}$ be the number of hyperedges present in this \textit{core hypergraph} $C =(X, Y)$. The set of hyperedges is denoted by 
$Y = \{e_1, e_2, e_3, \ldots, e_{m_{n-2}}\}$, where $e_i$ is the $i^{th}$ hyperedge in the core hypergraph $C$. 
Let $U = \{u_1, u_2, u_3, \ldots, u_n\}$ and $V = \{v_1, v_2, v_3, \ldots, v_n\}$ denote two disjoint set of vertices, each of them disjoint from $X$, the set of vertices in the core hypergraph $C$. Let $V^1 = (v_1, v_2, v_3, \ldots, v_n)$ denote the ordered set where $v_1 \prec v_2 \ldots \prec v_n$. For any $1 \leq p \leq n$, let $v_1^p=v_p$, $v_p^1 = v_1$ and $v_j^j = v_j$ for all $j \neq 1,p$. Let $V^p$ be the ordered set formed after swapping $v_1$ and $v_p$ of $V^1$. Then, $V^p = (v_1^p, v_2,\ldots, v_{p-1},v_p^1,v_{p+1} \ldots v_n)$ where $v_1^p \prec v_2^p \ldots \prec v_n^p$.

Consider one particular ordered set $V^p$. For each $1 \leq i \leq n$, we call $u_i \in U$ and $v_i^p \in V^p$ to be a {\it matching pair} of vertices in the set-ordered set pair ($U,V^p$). For any $K =  \{a_1, a_2, \ldots a_k\}$ which is a proper subset of $\{1, 2, \ldots, n\}$ such that $1 \leq a_1 < a_2 \ldots < a_k \leq n$, we denote $U_K = \{u_{a_1}, u_{a_2}, \ldots, u_{a_k}\}$ and $V_K^p = \{v_{a_1}^p, v_{a_2}^p, \ldots, v_{a_k}^p\}$ where $v_{a_1}^p$ is the matching vertex of $u_{a_1}$. We also define ${\bar{U}}_K = U \setminus U_K$ and ${\bar{V}^p_K} = V \setminus V_K^p$.

Now we define the construction of the non-$2$-colorable $n$-uniform hypergraph $H$. The vertex set of this hypergraph is $X \cup U \cup V$, and the set of hyperedges is the following:\\

$(i)$ $\{u_i\} \cup e_j \cup \{v_i\}$ for every pair of $i, j$ satisfying $2 \leq i \leq n$ and $1 \leq j \leq m_{n-2}$.

$(ii)$ $U$

$(iii)$ $U_K \cup {\bar{V}^p_K}$ for every $1 \leq p \leq n$ and every $K$ such that $|K|$ is odd, and $1 \leq |K| \leq \lfloor n/2 \rfloor$.

$(iv)$ $V_K^p \cup {\bar{U}}_K$ for every $1 \leq p \leq n$ and every $K$ such that $|K|$ is even, and $2 \leq |K| \leq \lfloor n/2 \rfloor$. \\

Note that the number of hyperedges formed in step $(i)$ is $(n-1) \cdot m_{n-2}$.
Recall that the step $(i)$ of the construction discussed in Section $2$ produces $n \cdot m_{n-2}$ hyperedges and we decrease this number to $(n-1) \cdot m_{n-2}$.

The step $(ii)$ of this construction produces a single hyperedge, while steps $(iii)$ and $(iv)$ produce a total of $\binom{n}{1} + \binom{n}{2} + \ldots + \binom{n}{\lfloor{n/2}\rfloor}$ hyperedges for every $p$ satisfying $1 \leq p \leq n$.
Therefore, we obtain:\\\\
$m(n) \leq n \cdot 2^{n-1} - (n-1) + (n-1) \cdot m(n-2)$, when $n$ is odd and,\\
$m(n) \leq n \cdot (2^{n-1} + {n \choose {n / 2}}/{2}) - (n-1) + (n-1) \cdot m(n-2)$, when $n$ is even.

\paragraph*{}

The two recurrence relations mentioned above can further be optimized because many of the hyperedges are repeated in the construction above. We observe below that we can improve these recurrence relations through a more careful analysis. 

\begin{observation}
Consider the hyperedges formed in the steps $(iii)$ and $(iv)$ of the construction. For any given $p$ satisfying $2 \leq p \leq n$, all the hyperedges which contain both the vertices $u_1$ and $u_p$, or both the vertices $v_1$ and $v_p$ are already formed in the steps $(iii)$ and $(iv)$ for $p = 1$.
\end{observation}

The number of hyperedges formed due to the set-ordered set pair ($U,V^1$) is $1 + \binom{n}{1} + \binom{n}{2} + \ldots + \binom{n}{\lfloor n/2 \rfloor} + (n-1) \cdot m_{n-2}$. The number of new hyperedges formed due to set-ordered set pair ($U,V^p$) for each $2 \leq p \leq n$, is at most $2\cdot\left(1 + \binom{n-2}{1} + \binom{n-2}{2} + \ldots + \binom{n-2}{\lfloor (n-2)/2 \rfloor} \right)$. This is due to the fact that the number of new hyperedges that contains one of the vertices $u_1$ and $u_p$ and are formed in the step $(iii)$ of the construction is $2\cdot\left(1 + \binom{n-2}{2} + \binom{n-2}{4} + \ldots + \binom{n-2}{\lfloor (n-2)/2 \rfloor} \right)$, similarly the number of new hyperedges that contains one of the vertices $v_1$ and $v_p$ and are formed in the step $(iv)$ of the construction is $2\cdot\left(\binom{n-2}{1} + \binom{n-2}{3} + \ldots + \binom{n-2}{\lfloor (n-2)/2 \rfloor} \right)$. Therefore, we obtain:\\\\
$m(n) \leq (n+1) \cdot 2^{n-2} + (n-1) \cdot m(n-2)$, when $n$ is odd and\\
$m(n) \leq (n+1) \cdot 2^{n-2} + {n \choose {n / 2}}/{2} + (n-1) \cdot
\left( m(n-2) + {n-2 \choose {(n-2)/2}} \right)$, when $n$ is even.\\

\noindent Let us now prove a Lemma which will be useful to prove that $H$ is not 2-colorable.

\begin{lemma}
\label{firstlemma}
Consider a coloring $\chi$ of the hypergraph $H$ having the property that there does not exist two different \textit{matching pairs} of vertices ($u_i,v_i^p$) and ($u_j,v_j^p$) in any of the set-ordered set pair ($U,V^p$) such that $u_i$ and $v_i^p$ are colored with Red, and, $u_j$ and $v_j^p$ are colored with Blue in $\chi$. Then, $\chi$ is not a proper $2$-coloring of $H$.
\end{lemma}
\begin{proof}
Assume for the sake of contradiction that $\chi$ is a proper $2$-coloring of $H$.
Without the loss of generality, let us assume that there is no \textit{matching pair} ($u_i,v_i^p$) in the set-ordered set pair ($U,V^p$) such that both $u_i$ and $v_i^p$ are colored by Red. Let $U_K$ denote the set of vertices in $U$ which are Red in this coloring, where $K \subseteq  \{1, 2, \ldots n\}$ denotes the set of indices of the vertices in the set $U_K$. Note that $U_K$ must be a proper subset of $U$, as $U$ is one of the hyperedges of $H$. Let us consider the following $2$ cases:

\begin{itemize}
\item \textbf{Case 1: $1 \leq |U_K| < \lfloor n/2\rfloor$}

If $|U_K|$ is even, then consider the hyperedge $V_K^p \cup {\bar{U}}_K$. All the vertices in this hyperedge are Blue, leading to a contradiction.

If $|U_K|$ is odd, then consider the hyperedge $U_K \cup {\bar{V}}_K^p$. At least one of the vertices, say 
$v_l^p \in {\bar{V}}_K^p$, must have been colored Blue in $\chi$. Consider the vertex $u_l$ such that ($u_l,v_l^p$) is a \textit{matching pair} of vertices in ($U$,$V^p$). It can be seen that the hyperedge $V_K^p \cup {\bar{U}}_K \cup \{v_l^p\} \setminus \{u_l\}$ is Blue, leading to a contradiction.

\item \textbf{Case 2: $\lfloor n/2\rfloor \leq |U_K| < n$}

If $n - |U_K|$ is odd, then consider the hyperedge $V_K^p \cup {\bar{U}}_K$. All the vertices in this hyperedge are Blue, leading to a contradiction.

If $n - |U_K|$ is even, then consider the hyperedge $U_K \cup {\bar{V}}_K^p$. At least one of the vertices, say $v_l^p \in {\bar{V}}_K^p$, must have been colored by Blue in $\chi$. Consider the vertex $u_l$ such that ($u_l,v_l^p$) is a \textit{matching pair} of vertices in ($U$,$V^p$). It can be seen that the hyperedge 
$V_K^p \cup {\bar{U}}_K \cup \{v_l^p\} \setminus \{u_l\}$ is Blue, leading to a contradiction. \qed
\end{itemize}
\end{proof}

\noindent In the following Lemma, we complete the proof for the non $2$-colorability of $H$.
\begin{lemma}
Our construction produces a hypergraph $H$ with a monochromatic hyperedge under any given Red-Blue coloring of its vertices.
\end{lemma}
\begin{proof}
Consider a Red-Blue coloring $\chi$. Without the loss of generality, let us assume that the \textit{core hypergraph} $C$ is monochromatic in Red under this coloring $\chi$. Note that the pair ($u_1,v_1$) is not combined with any hyperedge in the \textit{core hypergraph} to form a hyperedge. It can also be noted that any other pair ($u_i,v_i$), $i\not=1$, is combined with every hyperedge in the \textit{core hypergraph}, and therefore both $u_i$ and $v_i$ cannot be colored by Red.

\begin{itemize}
\item \textbf{Case 1:}
If either of $u_1$ or $v_1$ is colored with Blue, then there is no pair ($u_i,v_i$) such that $u_i$ and $v_i$ are both colored by Red. This satisfies the condition of Lemma \ref{firstlemma} for the set-ordered set pair ($U,V^1$). Hence, it follows from Lemma \ref{firstlemma} that the hypergraph $H$ is monochromatic under $\chi$.

\item \textbf{Case 2:}
If both $u_1$ and $v_1$ are colored with Red and if there does not exist any other pair of vertices ($u_i,v_i$) such that both $u_i$ and $v_i$ are colored with Blue, then it satisfies the condition of Lemma \ref{firstlemma} for the set-ordered set pair ($U,V^1$). Hence, it follows from Lemma \ref{firstlemma} that the hypergraph $H$ is monochromatic under $\chi$.

\item \textbf{Case 3:}
Let us assume that both $u_1$ and $v_1$ are colored with Red, and there exists at least one other pair ($u_i,v_i$), $i\not=1$, such that both $u_i$ and $v_i$ are colored with Blue. Consider the set-ordered set pair ($U,V^i$). It can be noted that there does not exist any \textit{matching pair} of vertices ($u_j,v_j^i$) in ($U,V^i$) such that both $u_j$ and $v_j^i$ are colored with Red. Hence, it follows from Lemma \ref{firstlemma} that the hypergraph $H$ is monochromatic under $\chi$.$\;\;\;\;$\qed\\
\end{itemize}
\end{proof}

We observe that using this construction, the first improvement is obtained when $n = 13$. The currently best known bound $m(13) \leq 360751$ is obtained from the recurrence relation provided by Abbott and Hanson \cite{Abbott} which is mentioned in Section $2$. Using the recurrence relation obtained from this construction, one can observe that $m(13) \leq 357892$. The next improvement is obtained when $n = 17$.

\section{Further improvement over the AHT construction}

For a given $n \geq 3$, let us consider an $(n-2)$-uniform hypergraph producing the best known upper bound for $m(n-2)$. Note that this hypergraph is the {\it core hypergraph} as described in the previous sections. Let $m_{n-2}$ be the number of hyperedges present in this \textit{core hypergraph} $C =(X, Y)$. The set of hyperedges is denoted by 
$Y = \{e_1, e_2, e_3, \ldots, e_{m_{n-2}}\}$, where $e_i$ is the $i^{th}$ hyperedge in the core hypergraph $C$. 
Let $U' = \{u'_1, u'_2, u'_3, \ldots, u'_{n-2}\}$ and $V' = \{v'_1, v'_2, v'_3, \ldots, v'_{n-2}\}$ denote two disjoint set of vertices, each of them disjoint from $X$, the set of vertices in the core hypergraph $C$.
Similarly, let $U = \{u_1, u_2, u_3, \ldots, u_n\}$ and $V = \{v_1, v_2, v_3, \ldots, v_n\}$ denote two disjoint set of vertices, each of them disjoint from $X$, $U'$ and $V'$. For each $1 \leq i \leq n$, we call $u_i$ and $v_i$ to be a pair of {\it matching vertices}. Similarly, for each $1 \leq i \leq n-2$, we call $u'_i$ and $v'_i$ to be a pair of {\it matching vertices}. For any $K =  \{a_1, a_2, \ldots a_k\}$ which is a proper subset of $\{1, 2, \ldots, n\}$ such that $1 \leq a_1 < a_2 \ldots < a_k \leq n$, we denote $U_K = \{u_{a_1}, u_{a_2}, \ldots, u_{a_k}\}$ and $V_K = \{v_{a_1}, v_{a_2}, \ldots, v_{a_k}\}$. We also define 
${\bar{U}}_K = U \setminus U_K$ and ${\bar{V}}_K = V \setminus V_K$. Similarly, for any $Q =  \{a_1, a_2, \ldots a_q\}$ which is a proper subset of $\{1, 2, \ldots, n-2\}$ such that $1 \leq a_1 < a_2 \ldots < a_q \leq n-2$, we denote $U'_Q = \{u'_{a_1}, u'_{a_2}, \ldots, u'_{a_q}\}$ and $V'_Q = \{v'_{a_1}, v'_{a_2}, \ldots, v'_{a_q}\}$. We also define 
${\bar{U'}}_Q = U' \setminus U'_Q$ and ${\bar{V'}}_Q = V' \setminus V'_Q$.

Let us now define the construction of the non-$2$-colorable $n$-uniform hypergraph $H$. The vertex set of this hypergraph is $X \cup U \cup V \cup U' \cup V'$, and the following hyperedges are present in $H$:\\

$(i)$ $\{u'_i\} \cup e_j \cup \{v'_i\}$ for every pair of $i, j$ satisfying $1 \leq i \leq n-2$ and $1 \leq j \leq m_{n-2}$.

$(ii)$ $U' \cup \{u_i\} \cup \{v_i\}$, for each $i$ satisfying $1 \leq i \leq n$.

$(iii)$ $U'_Q \cup {\bar{V'}}_Q \cup \{u_i\} \cup \{v_i\}$ for every $Q$ and $i$ such that $|Q|$ is odd, $1 \leq |Q| \leq \lfloor (n-2)/2 \rfloor$ and $1 \leq i \leq n$.

$(iv)$ $V'_Q \cup {\bar{U'}}_Q \cup \{u_i\} \cup \{v_i\}$ for every $Q$ and $i$ such that $|Q|$ is even, $2 \leq |Q| \leq \lfloor (n-2)/2 \rfloor$ and $1 \leq i \leq n$.

$(v)$ $U$

$(vi)$ $U_K \cup {\bar{V}}_K$ for every $K$ such that $|K|$ is odd, and $1 \leq |K| \leq \lfloor n/2 \rfloor$.

$(vii)$ $V_K \cup {\bar{U}}_K$ for every $K$ such that $|K|$ is even, and $2 \leq |K| \leq \lfloor n/2 \rfloor$. \\

Note that the Steps $(v)$, $(vi)$ and $(vii)$ in this construction are the same as the Steps $(ii)$, $(iii)$ and $(iv)$ of the construction of Abbott and Hanson \cite{Abbott} which is defined in Section $2$. The number of hyperedges formed in Step $(i)$ is $(n-2) \cdot m_{n-2}$. A total of $n \cdot \left( 1 + \binom{n-2}{1} + \binom{n-2}{2} + \ldots + \binom{n-2}{\lfloor (n-2)/2 \rfloor} \right)$ hyperedges are formed in Steps $(ii)$, $(iii)$ and $(iv)$. The number of hyperedges formed in Steps $(v)$, $(vi)$ and $(vii)$ is $1 + \binom{n}{1} + \binom{n}{2} + \ldots + \binom{n}{\lfloor n/2 \rfloor}$. Therefore, we obtain:\\\\
$m(n) \leq (n+4) \cdot 2^{n-3} + (n-2) \cdot m(n-2)$ when $n$ is odd and,\\ 
$m(n) \leq (n+4) \cdot 2^{n-3} +n \cdot {n-2 \choose {(n-2)/2}}/{2} + {n \choose {n / 2}}/{2} + (n-2) \cdot m(n-2)$ when $n$ is even.\\

\noindent The following Lemma proves that $H$ is not $2$-colorable.
\begin{lemma}
The construction defined above produces an $n$-uniform hypergraph $H$, which has a monochromatic hyperedge under any given Red-Blue coloring of its vertices.
\end{lemma}
\begin{proof}
Assume for the sake of contradiction that $\chi$ is a proper $2$-coloring of $H$. Without the loss of generality, let us assume that the non-$2$-colorable core hypergraph $C$ is monochromatic in Red under this coloring $\chi$. Note that none of the matching pair of vertices $u'_i$ and $v'_i$, $1 \leq i \leq (n-2)$, can have both their colors as Red as it causes a contradiction to our assumption due to Step $(i
)$ of the construction. Therefore, it follows from the $p=1$ case of Lemma \ref{firstlemma} that at least one of the $(n-2)$-uniform sets formed in the Steps $(ii), (iii)$ and $(iv)$ (after excluding the vertices $u_i$ and $v_i$, $1 \leq i \leq n$) of the construction must be monochromatic under the given coloring $\chi$. Without the loss of generality, let us assume that it is monochromatic in Blue. Then, none of the matching pair of vertices $u_i$ and $v_i$, $1 \leq i \leq n$, can have both their colors as Blue in $\chi$ as it causes a contradiction to our assumption. Then, it follows from the $p=1$ case of Lemma \ref{firstlemma} that all the hyperedges formed in Steps $(v), (vi)$ and $(vii)$ cannot be non-monochromatic under the above condition imposed on $\chi$. This contradicts the assumption that $\chi$ is a proper $2$-coloring of $H$. \qed\\
\end{proof}

\section{A construction that improves the upper bound for m(8)}

Let $A =(V_1, E_1)$ be a non-$2$-colorable $5$-uniform hypergraph that produces the best known upper bound for $m(5)$. Let $E_1 = \{ e^A_1, e^A_2, e^A_3, \ldots, e^A_{m(5)} \}$. Let $B =(V_2, E_2)$ be a Fano plane \cite{klein} such that the set of hyperedges is $E_2 = \{ e^B_1, e^B_2, e^B_3, \ldots, e^B_{m(3)} \}$. Let $C =(V_3, E_3)$ be another Fano plane \cite{klein} such that the set of hyperedges is $E_3 = \{ e^C_1, e^C_2, e^C_3, \ldots, e^C_{m(3)} \}$. Let $U = \{u_1, u_2, u_3, \ldots, u_8\}$ and $V = \{v_1, v_2, v_3, \ldots, v_8\}$ denote two disjoint set of vertices, each of them disjoint from $V_1$, $V_2$ and $V_3$. For each $1 \leq i \leq 8$, we call $u_i$ and $v_i$ to be a pair of {\it matching vertices}. For any $K =  \{a_1, a_2, \ldots a_k\}$ which is a proper subset of $\{1, 2, \ldots, 8\}$ such that $1 \leq a_1 < a_2 \ldots < a_k \leq 8$, we denote $U_K = \{u_{a_1}, u_{a_2}, \ldots, u_{a_k}\}$ and $V_K = \{v_{a_1}, v_{a_2}, \ldots, v_{a_k}\}$. We also define 
${\bar{U}}_K = U \setminus U_K$ and ${\bar{V}}_K = V \setminus V_K$.

Let us now define the construction of the non-$2$-colorable $8$-uniform hypergraph $H$. The vertex set of this hypergraph is $U \cup V \cup V_1 \cup V_2 \cup V_3$, and the following hyperedges are present in $H$:\\

$(i)$ $ e^A_j \cup e^B_l $ for every $j$ and $l$ satisfying $1 \leq j \leq m(5)$ and $1 \leq l \leq m(3)$.

$(ii)$ $ e^A_j \cup e^C_{l'} $ for every $j$ and $l'$ satisfying $1 \leq j \leq m(5)$ and $1 \leq l' \leq m(3)$.

$(iii)$ $ \{u_i\} \cup e^B_l \cup e^C_{l'} \cup \{v_i\} $ for every $i$, $l$ and $l'$ satisfying $1 \leq i \leq 8$, $1 \leq l, l' \leq m(3)$.

$(iv)$ $U$

$(v)$ $U_K \cup {\bar{V}}_K$ for every $K$ such that $|K|$ is odd, and $1 \leq |K| \leq 3$.

$(vi)$ $V_K \cup {\bar{U}}_K$ for every $K$ such that $|K|$ is even, and $2 \leq |K| \leq 4$. \\

Note that the Steps $(iv)$, $(v)$ and $(vi)$ in this construction are the same as the Steps $(ii)$, $(iii)$ and $(iv)$ of the construction of Abbott and Hanson \cite{Abbott} which is defined in Section $2$. The number of hyperedges formed in Step $(i)$ is $m(5) \cdot m(3) \leq 357$. The same number of hyperedges is formed in Step $(ii)$. The Step $(iii)$ produces $8 \cdot m(3) \cdot m(3) \leq 392$ hyperedges, while a single hyperedge is formed in Step $(iv)$. The Steps $(v)$ and $(vi)$ produce a total of $\binom{8}{1} + \binom{8}{2} + \binom{8}{3} + \binom{8}{4}$ hyperedges. Adding these values, we obtain $m(8) \leq 1269$.\\

\noindent The following Lemma proves that $H$ is not $2$-colorable.
\begin{lemma}
The construction defined above produces an $8$-uniform hypergraph $H$, which has a monochromatic hyperedge under any given Red-Blue coloring of its vertices.
\end{lemma}
\begin{proof}
Assume for the sake of contradiction that $\chi$ is a proper $2$-coloring of $H$. Without the loss of generality, let us assume that the non-$2$-colorable $5$-uniform hypergraph $A$ is monochromatic in Red under this coloring $\chi$ and let $e^A_j \in E_1$ be the hyperedge that is monochromatic in Red. It can be noted from Step $(i)$ of the construction that the non-$2$-colorable $3$-uniform hypergraph $B$ must contain a hyperedge $e^B_l \in E_2$ which is monochromatic in Blue. Similarly, it can be seen from Step $(ii)$ of the construction that the non-$2$-colorable $3$-uniform hypergraph $C$ contains a hyperedge $e^C_{l'} \in E_3$ which is monochromatic in Blue. Since both the hyperedges $e^B_l$ and $e^C_{l'}$ are Blue, it follows from Step $(iii)$ that any of the \textit{matching pair} of vertices $u_i$ and $v_i$, 
$u_i \in U$, $v_i \in V$, $1 \leq i \leq 8$, cannot have both their colors as Blue in $\chi$. It follows from the
$p=1$ case of Lemma \ref{firstlemma} that all the hyperedges formed in the Steps $(iv), (v)$ and $(vi)$ cannot be non-monochromatic under the above condition imposed on $\chi$. This contradicts the assumption that
$\chi$ is a proper $2$-coloring of $H$. \qed
\\
\end{proof}

\begin{table}
\caption{The improved upper bound on $m(n)$ for small values of $n$.}
\begin{center}
\begin{tabular}{|c|c|c|}
\hline
$n$ & $m(n)$ & Obtained by the recurrence\\
\hline
1& $m(1) = 1$&Single vertex\\
2&$m(2) = 3$&Triangle graph\\
3&$m(3) = 7$&Fano plane \cite{klein}\\
4&$m(4) = 23$&$m(4) \leq 2^3 + {4 \choose 2} / {2} + 4 \cdot m(2)$ and $\cite{patric}$\\
5&$m(5) \leq 51$&$m(5) \leq 2^4 + 5 \cdot m(3)$\\
6&$m(6) \leq 147$&$m(6) \leq m(2) \cdot m(3)^2$\\
7&$m(7) \leq 421$&$m(7) \leq 2^6 + 7 \cdot m(5)$\\
\textbf{8}&$\textbf{m(8)} \leq \textbf{1269}$&$m(8) \leq 2 \cdot m(5) \cdot m(3) + 8 \cdot m(3) \cdot m(3) + 2^7 + {8 \choose 4} / {2} $\\
9&$m(9) \leq 2401$&$m(9) \leq m(3) \cdot m(3)^3$\\
10&$m(10) \leq 7803$&$m(10) \leq m(2) \cdot m(5)^2$\\
\textbf{11}&$\textbf{m(11)} \leq \textbf{25449}$&$m(11) \leq 15 \cdot 2^8 + 9 \cdot m(9)$\\
12&$m(12) \leq 64827$&$m(12) \leq m(2) \cdot m(6)^2$\\
\textbf{13}&$\textbf{m(13)} \leq \textbf{297347}$&$m(13) \leq 17 \cdot 2^{10} + 11 \cdot m(11)$ \\
14&$m(14) \leq 531723$&$m(14) \leq m(2) \cdot m(7)^2$\\
15&$m(15) \leq 857157$&$m(15) \leq m(5) \cdot m(3)^5$\\
16&$m(16) \leq 5378763$&$m(16) \leq m(2) \cdot m(8)^2$\\
\textbf{17}&$\textbf{m(17)} \leq \textbf{13201419}$&$m(17) \leq 21 \cdot 2^{14} + 15 \cdot m(15)$ \\

\hline
\end{tabular}
\end{center}
\end{table}

In Table $2$, we show the improved upper bounds on $m(n)$ for small values of $n$. Comparing Tables $1$ and $2$, one can observe that the first few improvements are obtained when $n = 8$, $11$, $13$ and $17$.

\section{Conclusion}

Although the constructions given in this paper provide better upper bounds on $m(n)$ for small values of $n$, it gives a weak upper bound for a general $n$. Moreover, when $n<8$, the upper bounds on $m(n)$ provided by our constructions do not improve the upper bounds provided by  previously known constructions. The first interesting open problem is to improve the upper bound on $m(5)$ or to prove that $m(5)=51$.

\section*{Acknowledgement}

This work has been supported by Ramanujan Fellowship, Department of Science and Technology, Government of India, grant number SR/S2/RJN-87/2011.

\end{document}